\documentclass{aic}

\aicAUTHORdetails{%
	title = {Planar Graphs Have Bounded Nonrepetitive Chromatic Number}, 
	author = {Vida Dujmovi{\'c}, Louis Esperet, Gwena\"{e}l Joret, Bartosz Walczak, David R. Wood},
	plaintextauthor = {Vida Dujmovic, Louis Esperet, Gwenael Joret, Bartosz Walczak, David R. Wood},
}   

\aicEDITORdetails{%
	year={2020},
	number={5},
	received={11 April 2019},
	published={6 March 2020},
	updated={21 January 2022},
}

\usepackage{amsfonts,enumerate,underscore,mathtools,paralist}
\usepackage[noabbrev,capitalise]{cleveref}
\usepackage[longnamesfirst,numbers,sort&compress]{natbib}
\usepackage{hypernat} 
\makeatletter
\def\NAT@spacechar{~}
\makeatother
\allowdisplaybreaks

\renewcommand{\ge}{\geqslant}
\renewcommand{\le}{\leqslant}
\renewcommand{\geq}{\geqslant}
\renewcommand{\leq}{\leqslant}
\newcommand{\arXiv}[1]{arXiv:\,\href{https://arxiv.org/abs/#1}{#1}}
\newcommand{\msn}[1]{MR:\,\href{https://www.ams.org/mathscinet-getitem?mr=MR#1}{#1}}
\newcommand{\doi}[1]{\url{https://dx.doi.org/#1}}
\theoremstyle{plain}
\newtheorem{theorem}{Theorem}
\newtheorem{lemma}[theorem]{Lemma}
\newtheorem{corollary}[theorem]{Corollary}
\begin{document}
	
	\begin{frontmatter}[classification=text]
\title{Planar Graphs Have Bounded\\ Nonrepetitive Chromatic Number\footnote{This revision of the original published paper corrects two errors. First,  ``$K_k+$'' has been added to \cref{AlmostEmbeddableStructure} and ``$k+$'' has been added in several places in \cref{Minors}. Second, the original published paper  erroneously stated that \citet{DJMMUW} proved \cref{GenusProduct}; they actually proved the weakening of \cref{GenusProduct} with ``treewidth at most\/ $3$'' replaced by ``treewidth at most\/ $4$''. This result leads to a bound of $\pi(G)\leq 1024\max\{2g,3\}$.}}

\author[vd]{Vida Dujmovi{\'c}\thanks{Research  supported by NSERC and the Ontario Ministry of Research and Innovation.}}
\author[le]{Louis Esperet\thanks{Partially supported by ANR Projects GATO (\textsc{anr-16-ce40-0009-01}) and GrR (\textsc{anr-18-ce40-0032}).}}
\author[gj]{Gwena\"el Joret\thanks{Research supported by an ARC grant from the Wallonia-Brussels Federation of Belgium.}}
\author[bw]{Bartosz Walczak\thanks{Research partially supported by National Science Centre of Poland grant 2015/17/D/ST1/00585.}}
\author[dw]{David R. Wood\thanks{Research supported by the Australian Research Council.}}
		
\begin{abstract}
A colouring of a graph is \emph{nonrepetitive} if for every path of even order, the sequence of colours on the first half of the path is different from the sequence of colours on the second half. We show that planar graphs have nonrepetitive colourings with a bounded number of colours, thus proving a conjecture of Alon, Grytczuk, Ha{\l}uszczak and Riordan (2002). We also generalise this result for graphs of bounded Euler genus, graphs excluding a fixed minor, and graphs excluding a fixed topological minor.
\end{abstract}

\end{frontmatter}
	
\section{Introduction}
\label{Introduction}

A vertex colouring of a graph is \emph{nonrepetitive} if there is no path for which the first half of the path is assigned the same sequence of colours as the second half.  More precisely, a $k$-\emph{colouring} of a graph $G$ is a function $\phi$ that assigns one of $k$ colours to each vertex of $G$.  A path $(v_1,v_2,\dots,v_{2t})$ of even order in $G$ is \emph{repetitively} coloured by $\phi$ if $\phi(v_i)=\phi(v_{t+i})$ for $i\in\{1,\dots,t\}$. A colouring $\phi$ of $G$ is \emph{nonrepetitive} if no path of $G$ is repetitively coloured by $\phi$. Observe that a nonrepetitive colouring is \emph{proper}, in the sense that adjacent vertices are coloured differently. The
\emph{nonrepetitive chromatic number} $\pi(G)$ is the minimum integer $k$ such that $G$ admits a nonrepetitive $k$-colouring.

The classical result in this area is by \citet{Thue06}, who proved in 1906 that every path is nonrepetitively 3-colourable.
Starting with the seminal work of \citet{AGHR02}, nonrepetitive colourings of general graphs have recently been widely studied; see the surveys \citep{Gryczuk-IJMMS07,Grytczuk-DM08,CSZ,Grytczuk-Berge,Skrabu15} and other references  \citep{GMP14,HJ-DM11,Currie-EJC02,HJSS11,Grytczuk-DM08,DJKW16,PZ09,DFJW13,BreakingRhythm,BaratWood-EJC08,BV-NonRepVertex07,DMW17,BK-AC04,GKM11,BGKNP-NonRepTree-DM07,Currie-TCS05,NOW11,KP-DM08,FOOZ,Pegden11,BV-NonRepEdge08,Grytczuk-EJC02,AGHR02,GPZ11,DS09,BC13,KT17,BDMR17,Thue06,GlJKM16,GKZ16,WW16,ZZ16,KPZ14,CG07,AG05,Przybyo14,Przybyo13,JS12,JS09,WW19,SS15a,PSFT18,Gryczuk-IJMMS07,CSZ,Grytczuk-Berge,Skrabu15}.

Several graph classes are known to have bounded nonrepetitive chromatic number. In particular, 
cycles are nonrepetitively 3-colourable (except for a finite number of exceptions) \citep{Currie-EJC02}, 
trees are nonrepetitively 4-colourable \citep{BGKNP-NonRepTree-DM07,KP-DM08}, outerplanar graphs are nonrepetitively $12$-colourable \citep{KP-DM08,BV-NonRepVertex07}, and more generally, every graph with treewidth $k$ is nonrepetitively $4^k$-colourable \citep{KP-DM08}. Graphs with maximum degree $\Delta$ are nonrepetitively $O(\Delta^2)$-colourable \citep{AGHR02,Gryczuk-IJMMS07,HJ-DM11,DJKW16}, and graphs excluding a fixed immersion have bounded nonrepetitive chromatic number \citep{WW19}. 

It is widely recognised that the most important open problem in the field of nonrepetitive graph colouring is whether planar graphs have bounded nonrepetitive chromatic number. It was first asked by \citet{AGHR02}. The best known lower bound is $11$, due to Ochem (see \citep{DFJW13}). The best known upper bound is $O(\log n)$ where $n$ is the number of vertices, due to \citet*{DFJW13}. Note that several works have studied colourings of planar graphs in which only facial paths are required to be nonrepetitively coloured \citep{HJSS11,Przybyo14,Przybyo13,BC13,JS12,JS09,BDMR17}.

This paper proves that planar graphs have bounded nonrepetitive chromatic number.

\begin{theorem}
\label{PlanarPi}
Every planar graph\/ $G$ satisfies\/ $\pi(G) \leq 768$.
\end{theorem}

We generalise this result for graphs of bounded Euler genus, for graphs excluding any fixed minor, and for graphs excluding any fixed topological minor.\footnote{The \textit{Euler genus} of the orientable surface with $h$ handles is $2h$. The \textit{Euler genus} of the non-orientable surface with $c$ cross-caps is $c$. The \textit{Euler genus} of a graph $G$ is the minimum integer $k$ such that $G$ embeds in a surface of Euler genus $k$. Of course, a graph is planar if and only if it has Euler genus 0; see \citep{MoharThom} for more about graph embeddings in surfaces. A graph $X$ is a \textit{minor} of a graph $G$ if a graph isomorphic to $X$ can be obtained from a subgraph of $G$ by contracting edges. A graph $X$ is a \emph{topological minor} of a graph $G$ if a subdivision of $X$ is a subgraph of $G$. If $G$ contains $X$ as a topological minor, then $G$ contains $X$ as a minor. If $G$ contains no $X$ minor, then $G$ is \emph{$X$-minor-free}. If $G$ contains no $X$ topological minor, then $G$ is \emph{$X$-topological-minor-free.} }  
The result for graphs excluding a fixed minor confirms a conjecture of \citet{Gryczuk-IJMMS07,Grytczuk-DM08}.  

\begin{theorem}
\label{GenusPi}
Every graph\/ $G$ with Euler genus\/ $g$ satisfies\/ $\pi(G) \leq 256 \max\{2g,3\}$.
\end{theorem}

\begin{theorem}
\label{MinorPi}
For every graph\/ $X$, there is an integer\/ $k$ such that every\/ $X$-minor-free graph\/ $G$ satisfies\/ $\pi(G) \leq k$. 
\end{theorem}

\begin{theorem}
\label{TopoMinorPi}
For every graph\/ $X$, there is an integer\/ $k$ such that every\/ $X$-topological-minor-free graph\/ $G$ satisfies\/ $\pi(G) \leq k$. 
\end{theorem}

The proofs of \cref{PlanarPi,GenusPi} are given in \cref{PlanarGenus}, and 
the proofs of \cref{MinorPi,TopoMinorPi} are given in \cref{Minors}.
Before that in \cref{tools} we introduce the tools used in our proofs, namely 
so-called strongly nonrepetitive colourings, tree-decompositions and treewidth, and strong products.
With these tools in hand, the above theorems quickly follow from recent results of \citet{DJMMUW} that show that planar graphs and other graph classes are subgraphs of certain strong products. 

\section{Tools}
\label{tools}

Undefined terms and notation can be found in \citep{Diestel4}.

\subsection{Strongly Nonrepetitive Colourings}
\label{StronglyNonrepetitive}

A key to all our proofs is to consider a strengthening of nonrepetitive colouring defined below. 

For a graph $G$, a \emph{lazy walk} in $G$ is a sequence of vertices
$v_1,\ldots,v_k$ such that for each $i\in\{1,\dots,k\}$, either 
$v_iv_{i+1}$ is an edge of $G$, or $v_i=v_{i+1}$. A lazy walk can be thought of as 
a walk in the graph obtained from $G$ by adding a loop at each vertex. 
For a colouring $\phi$ of $G$, a lazy walk $v_1,\ldots,v_{2k}$ is \emph{$\phi$-repetitive} if $\phi(v_i)=\phi(v_{i+k})$ for each $i\in\{1,\dots,k\}$.

A colouring $\phi$ is \emph{strongly nonrepetitive} if for every $\phi$-repetitive lazy walk $v_1,\ldots,v_{2k}$, there exists $i\in\{1,\dots,k\}$ such that $v_i=v_{i+k}$. Let $\pi^*(G)$ be the minimum number of colours in a strongly
nonrepetitive colouring of $G$. 
Since a path has no repeated vertices, every strongly nonrepetitive colouring is nonrepetitive, and thus
$\pi(G)\le \pi^*(G)$ for every graph $G$.

\subsection{Layerings}
\label{Layerings}

A \emph{layering} of a graph $G$ is a partition $(V_0,V_1,\dots)$ of $V(G)$ such that for every edge $vw\in E(G)$, if $v\in V_i$ and $w\in V_j$, then $|i-j| \leq 1$. If $r$ is a vertex in a connected graph $G$ and $V_i$ is the set of vertices at distance exactly $i$ from $r$ in $G$ for all $i\geq 0$, then the layering $(V_0,V_1,\dots)$ is called a \emph{BFS layering} of $G$.

Consider a layering $(V_0,V_1,\dots)$ of a graph $G$. Let $H$ be a
connected component of $G[V_i\cup V_{i+1}\cup \cdots]$, for some $i\ge
1$. The \emph{shadow} of  $H$ is the set of vertices in $V_{i-1}$
adjacent to  some vertex in $H$. The layering is
\emph{shadow-complete} if every shadow is a clique. This concept was
introduced by \citet{KP-DM08} and implicitly by \citet{DMW05}.

\smallskip

We will need the following result.

\begin{lemma}[\citep{KP-DM08,DMW05}]
\label{cla:sc}
Every BFS-layering of a connected chordal graph is shadow-complete.
\end{lemma}

\subsection{Treewidth}
\label{Treewidth}

A \emph{tree-decomposition} of a graph $G$ consists of a collection $\{B_x\subseteq V(G) : x\in V(T)\}$ of subsets of $V(G)$, called \emph{bags}, indexed by the vertices of a tree $T$, and with the following properties:
\begin{compactitem}
\item for every vertex $v$ of $G$, the set $\{x\in V(T) : v\in B_x\}$ induces a non-empty (connected) subtree of $T$, and
\item for every edge $vw$ of $G$, there is a vertex $x\in V(T)$ for which $v,w\in B_x$.
\end{compactitem}
The \emph{width} of such a tree-decomposition is $\max\{|B_x|:x\in V(T)\}-1$. 
The \emph{treewidth} of a graph $G$ is the minimum width of a tree-decomposition of $G$.
Tree-decompositions were introduced by \citet{RS-II}.
Treewidth measures how similar a given graph is to a tree, and is particularly important in structural and algorithmic graph theory.

\citet{BV-NonRepVertex07} and \citet{KP-DM08} independently proved that graphs of bounded treewidth have bounded nonrepetitive chromatic number. Specifically, \citet{KP-DM08} proved that every graph with treewidth $k$ is  nonrepetitively $4^k$-colourable, which is the best known bound. \cref{TreewidthPi*} below strengthens this result. The proof is almost identical to that of \citet{KP-DM08} and depends on the following lemma. A lazy walk $v_1,\ldots,v_{2k}$ is \emph{boring} if $v_i=v_{i+k}$ for each $i\in\{1,\dots,k\}$. 

\begin{lemma}[\citep{KP-DM08}]
\label{cla:path}
Every path\/ $P$ has a\/ $4$-colouring\/ $\phi$ such that every\/ $\phi$-repetitive lazy walk is boring.
\end{lemma}

\begin{theorem}
\label{TreewidthPi*}
For every graph\/ $G$ of treewidth at most\/ $k\ge 0$, we have\/ $\pi^*(G)\le 4^k$.
\end{theorem}

\begin{proof}
The proof proceeds by induction on $k$. If $k=0$, then $G$ has
no edges, so assigning the same colour to all the vertices gives a strongly nonrepetitive colouring.
For the rest of the proof, assume that $k\ge 1$. 
Consider a tree-decomposition of $G$ of width at most $k$. 
By adding edges if necessary, we may assume that every bag of the tree-decomposition is a clique. 
Thus, $G$ is connected and chordal, with clique-number at most $k+1$.

Let $(V_0,V_1,\ldots)$ be a BFS-layering of $G$.
We refer to $V_i$ as the set of vertices at \emph{depth} $i$.
Note that the subgraph $G[V_i]$ of $G$ induced by each layer $V_i$ has
treewidth at most $k-1$.\footnote{This is clear for $i=0$ (since $k\geq 1$), 
and for $i\geq 1$ this follows from the fact that the graph $G[V_i]$ plus a universal vertex is a minor of $G$ (contract $V_0 \cup \cdots \cup V_{i-1}$ into a single vertex and remove $V_{i+1}, V_{i+2}, \dots$), and thus has treewidth at most $k$. Since removing a universal vertex decreases the treewidth by exactly one, 
it follows that $G[V_i]$ has treewidth at most $k-1$.} 
Thus the spanning subgraph $H$ of $G$
induced by all edges whose endpoints have the same depth also has
treewidth at most $k-1$. By the induction hypothesis, $H$ has a strongly
nonrepetitive colouring $\phi_1$
with $4^{k-1}$ colours. The graph $P$ obtained from $G$ by contracting
each set $V_i$ (which might not induce a connected graph) into a
single vertex $x_i$ is a path, and thus, by \cref{cla:path}, $P$
has a $4$-colouring $\phi_2$ such that every
$\phi_2$-repetitive walk is boring. For each $i\geq 0$ and each vertex $u\in V_i$,
set $\phi(u):=(\phi_1(u), \phi_2(x_i))$. The colouring $\phi$ of $G$ uses at most
$4\cdot 4^{k-1}=4^k$ colours. 

We now prove that $\phi$ is strongly nonrepetitive. 
Let $W$ be a $\phi$-repetitive lazy walk $v_1,\ldots,v_{2k}$.
Our goal is to prove that $v_j=v_{j+k}$ for some $j\in\{1,\ldots,k\}$. 
Let $d$ be the minimum depth of a vertex in $W$.

Let $W'$ be the sequence of vertices obtained from $W$ by removing all vertices at depth greater than $d$. 
We claim that $W'$ is a lazy walk. To see this, consider vertices $v_i,v_{i+1}, \ldots, v_{i+t}$ of $W$ such that $v_i$ and $v_{i+t}$ have depth
$d$ but $v_{i+1}, \ldots, v_{i+t-1}$ all have depth greater than $d$; thus, $v_{i+1}, \ldots, v_{i+t-1}$ were removed when constructing $W'$.
Then, the vertices $v_{i+1}, \ldots, v_{i+t-1}$ lie
in a connected component of the graph induced by the vertices of depth
greater than $d$, thus it follows that $v_i$
and $v_{i+t}$ are adjacent or equal by \cref{cla:sc}.  
This shows that $W'$ is a lazy walk in $G[V_d]$.

The projection of $W$ on $P$ is a $\phi_2$-repetitive lazy walk in $P$, and is thus boring by
\cref{cla:path}. It follows that the vertices $v_j$
and $v_{j+k}$ of $W$ have the same depth for every $j \in \{1, \dots, k\}$. 
In particular, $v_j$ was removed from $W'$ if and only if $v_{j+k}$ was. 
Hence, there are indices $1 \leq i_1 < i_2 < \cdots < i_{\ell} \leq k$ 
such that $W'=v_{i_1}, v_{i_2}, \dots, v_{i_{\ell}}, v_{i_1+k}, v_{i_2+k}, \dots, v_{i_{\ell}+k}$. 
Since $W$ is $(\phi_1,\phi_2)$-repetitive, it follows that $W'$ is
also $(\phi_1,\phi_2)$-repetitive and in particular $W'$ is $\phi_1$-repetitive. 
By the definition of $\phi_1$, there is an index $i_r$ such that $v_{i_r}=v_{i_r+k}$, which completes the proof.
\end{proof}

\subsection{Strong Products}
\label{Products}

The \emph{strong product} of graphs $A$ and $B$, denoted by $A\boxtimes B$, is the graph with vertex set $V(A)\times V(B)$, where distinct vertices $(v,x),(w,y)\in V(A)\times V(B)$ are adjacent if 
(1) $v=w$ and $xy\in E(B)$, or 
(2) $x=y$ and $vw\in E(A)$, or  
(3) $vw\in E(A)$ and $xy\in E(B)$. 
Nonrepetitive colourings of graph products have been studied in \citep{KP-DM08,KPZ14,BaratWood-EJC08,PSFT18}. Indeed, \citet{KP-DM08} note that their method shows that the strong product of $k$ paths is nonrepetitively $4^k$-colourable, which is a precursor to the following results.

\begin{lemma}\label{lem:strp}
  Let\/ $H$ be a graph with an\/ $\ell$-colouring\/ $\phi_2$ such that every\/
  $\phi_2$-repetitive lazy walk is boring. 
  For every graph\/ $G$, we have\/ $\pi^*(G \boxtimes H)\le \ell \,\pi^*(G)$.
\end{lemma}

\begin{proof}
Consider a strongly nonrepetitive colouring $\phi_1$ of $G$ with $\pi^*(G)$
colours.  For any two vertices $u \in V(G)$ and $v \in V(H)$, we define the colour $\phi(u,v)$
of the vertex $(u,v)\in V(G\boxtimes H)$ by $\phi(u,v) :=(\phi_1(u),\phi_2(v))$.
We claim that this is a strongly nonrepetitive colouring of
$G\boxtimes H$. To
see this, consider a $\phi$-repetitive lazy walk $W=(u_1,v_1), \ldots, 
(u_{2k},v_{2k})$ in $G\boxtimes H$. By the definition of the strong
product and the definition of $\phi$, the projection
$W_G=u_1,u_2,\ldots,u_{2k}$ of $W$ on $G$ is a $\phi_1$-repetitive 
lazy walk in $G$ and the projection $W_H=v_1,v_2,\ldots,v_{2k}$ of $W$ on $H$ is a
$\phi_2$-repetitive lazy walk in $H$. By the definition of $\phi_1$, there
is an index $i$ such that $u_i=u_{i+k}$. By the definition of $\phi_2$,
we have $v_j=v_{j+k}$ for every $j \in \{1, \dots, k\}$.
In particular, $v_i=v_{i+k}$ and $(u_i,v_i)=(u_{i+k},v_{i+k})$, which
completes the proof.
\end{proof}

Applying \cref{cla:path}, we obtain the following immediate corollary.

\begin{corollary}\label{ProductPi*}
For every graph\/ $G$ and every path\/ $P$, we have\/ $\pi^*(G \boxtimes P)\le 4 \pi^*(G)$.
\end{corollary}

By taking $H=K_\ell$ and a proper $\ell$-colouring $\phi_2$ of $K_\ell$, 
we obtain the following corollary of \cref{lem:strp}.

\begin{corollary}\label{BlowupPi*}
For every graph\/ $G$ and every integer\/ $\ell\geq 1$, we have\/ $\pi^*(G \boxtimes K_{\ell})\le \ell\,\pi^*(G)$.
\end{corollary}

\section{Planar Graphs and Graphs of Bounded Genus}
\label{PlanarGenus}

The following recent result  by \citet{DJMMUW} is a key theorem. 

\begin{theorem}[\citep{DJMMUW}]
\label{PlanarProductTreewidth3}
Every planar graph is a subgraph of\/ $H \boxtimes P \boxtimes K_{3}$ for some graph\/ $H$ with treewidth at most\/ $3$ and some path\/ $P$.
\end{theorem}

\cref{ProductPi*,TreewidthPi*,PlanarProductTreewidth3} imply that for every planar graph $G$, 
$$\pi(G) \leq \pi^*(G) \leq \pi^*( H \boxtimes P \boxtimes K_{3}) \leq 3\, \pi^*(H\boxtimes P) \leq 3 \cdot 4\,  \pi^*(H) \leq 3 \cdot 4  \cdot 4^3 = 768,$$
which proves \cref{PlanarPi}.

\medskip
For graphs of bounded Euler genus, \citet{DHHW21} proved the following strengthening of \cref{PlanarProductTreewidth3}.

\begin{theorem}[\citep{DHHW21}] 
\label{GenusProduct} 
Every graph\/ $G$ of Euler genus\/ $g$ is a subgraph of\/ $H \boxtimes P \boxtimes K_{\max\{2g,3\}}$ for some graph\/ $H$ with treewidth at most\/ $3$  and some path\/ $P$. 
\end{theorem}

\cref{ProductPi*,TreewidthPi*,GenusProduct} imply that for every graph $G$ with Euler genus $g$, 
\begin{align*}
\pi(G) \leq \pi^*(G) 
\leq \pi^*( H \boxtimes P \boxtimes K_{\max\{2g,3\}} ) 
\leq \max\{2g,3\}  \cdot \pi^*( H\boxtimes P ) 
& \leq \max\{2g,3\}  \cdot 4 \cdot \pi^*( H ) \\
& \leq  \max\{2g,3\}  \cdot 4^4\\
& =  256 \max\{2g,3\},
\end{align*}
which proves \cref{GenusPi}.

\section{Excluded Minors}
\label{Minors}

Our results for graphs excluding a minor depend on the following version of the graph minor structure theorem of \citet{RS-XVI}. A tree-decomposition $(B_x :x\in V(T))$ of a graph $G$ is \emph{$r$-rich} if $B_x\cap B_y$ is a clique in $G$ on at most $r$ vertices, for each edge $xy\in E(T)$. 

\begin{theorem}[\citep{DMW17}]
\label{ProduceRichDecomp}
For every graph\/ $X$, there are integers\/ $r\geq 1$ and\/ $k\geq 1$ such that every\/ $X$-minor-free graph\/ $G_0$ is a spanning subgraph of a graph\/ $G$ that has an\/ $r$-rich tree-decomposition such that each bag induces a\/ $k$-almost-embeddable subgraph of\/ $G$.
\end{theorem}

We omit the definition of $k$-almost embeddable from this paper, since we do not need it. All we need to know is the following theorem of \citet{DJMMUW}, where $A+B$ is the complete join of graphs $A$ and $B$.

\begin{theorem}[\citep{DJMMUW}] 
\label{AlmostEmbeddableStructure}
Every\/ $k$-almost embeddable graph is a subgraph of\/ $K_k + ( H\boxtimes P \boxtimes K_{\max\{6k,1\}} )$ for some graph\/ $H$ with treewidth at most\/ $11k+10$.
\end{theorem}

Observe that $\pi^*(G+K_k)=\pi^*(G)+k$ for every graph $G$ and integer $k\ge 0$. Thus \cref{AlmostEmbeddableStructure,ProductPi*,TreewidthPi*} imply that for every $k$-almost embeddable graph $G$, 
\begin{equation}
\label{AlmostEmbeddablePi}
\pi(G) \leq \pi^*(G) 
\leq \pi^*( K_k + ( H\boxtimes P \boxtimes K_{\max\{6k,1\}}) )
\leq k+ 6k \, \pi^*(H\boxtimes P) 
\leq k+ 6k \cdot 4  \pi^*(H) 
\leq k+ 6k \cdot 4^{11(k+1)}.
\end{equation}

\citet{DMW17} proved the following lemma, which generalises a result of \citet{KP-DM08}. 

\begin{lemma}[\citep{DMW17}] 
\label{RichColour}
Let\/ $G$ be a graph that has an\/ $r$-rich tree-decomposition such that the subgraph induced by each bag is nonrepetitively\/ $c$-colourable. Then\/ $\pi(G) \leq  c\,4^r$.
\end{lemma}


%

\cref{ProduceRichDecomp,RichColour} and \eqref{AlmostEmbeddablePi} with $c=k+6k \cdot 4^{11(k+1)}$ imply that for every graph $X$ and every $X$-minor-free graph $G$, 
\begin{equation*}
\pi(G) \leq \pi^*(G) \leq (k + 6k \cdot 4^{11(k+1)} )  4^r ,
\end{equation*}
which implies \cref{MinorPi} since $k$ and $r$ depend only on $X$. 

To obtain our result for graphs excluding a fixed topological minor, we use the following version of the structure theorem of \citet{GM15}.

%

\begin{theorem}[\citep{DMW17}]
\label{TopoProduceRichDecomp}
For every graph\/ $X$, there are integers\/ $r\geq 1$ and\/ $k\geq 1$ such that every\/ $X$-topological-minor-free graph\/ $G_0$ is a spanning subgraph of a graph\/ $G$ that has an\/ $r$-rich tree-decomposition such that the subgraph induced by each bag is\/ $k$-almost-embeddable or has at most\/ $k$ vertices with degree greater than\/ $k$.
\end{theorem}

\citet{AGHR02} proved that graphs with maximum degree $\Delta$ are nonrepetitively $O(\Delta^2)$-colourable. The best known bound is due to \citet{DJKW16}.

\begin{theorem}[\cite{DJKW16}]
\label{NonRepColourDegree}
Every graph with maximum degree\/ $\Delta\geq 2$ is nonrepetitively\/ $(\Delta^2+O(\Delta^{5/3}))$-colourable.
\end{theorem}

\cref{NonRepColourDegree} implies that if a graph has at most $k$ vertices with degree greater than $k$, then it is nonrepetitively 
$c'$-colourable for some constant $c'= k^2+O(k^{5/3}) + k$. 
\cref{TopoProduceRichDecomp,RichColour} and \eqref{AlmostEmbeddablePi} with 
$c=\max\{ k+ 6k \cdot 4^{11(k+1)},  c'  \}$ 
imply that for every graph $X$, every $X$-topological-minor-free graph $G$ satisfies 
$\pi(G) \leq \pi^*(G) \leq c \cdot 4^r$, which implies \cref{TopoMinorPi}, since $c$ and $r$ depend only on $X$. 


\section*{Acknowledgements}

This research was completed at the Workshop on Graph Theory held at Bellairs Research Institute in April 2019. Thanks to the other workshop participants for creating a productive working atmosphere. Thanks to Andr\'e K{\"u}ndgen and  Jaros{\l}aw Grytczuk for helpful comments on an early draft. 

\label{lastpage}

  \let\oldthebibliography=\thebibliography
  \let\endoldthebibliography=\endthebibliography
  \renewenvironment{thebibliography}[1]{%
    \begin{oldthebibliography}{#1}%
      \setlength{\parskip}{0.0ex}%
      \setlength{\itemsep}{0.0ex}%
  }{\end{oldthebibliography}}


\def\soft#1{\leavevmode\setbox0=\hbox{h}\dimen7=\ht0\advance \dimen7
	by-1ex\relax\if t#1\relax\rlap{\raise.6\dimen7
		\hbox{\kern.3ex\char'47}}#1\relax\else\if T#1\relax
	\rlap{\raise.5\dimen7\hbox{\kern1.3ex\char'47}}#1\relax \else\if
	d#1\relax\rlap{\raise.5\dimen7\hbox{\kern.9ex \char'47}}#1\relax\else\if
	D#1\relax\rlap{\raise.5\dimen7 \hbox{\kern1.4ex\char'47}}#1\relax\else\if
	l#1\relax \rlap{\raise.5\dimen7\hbox{\kern.4ex\char'47}}#1\relax \else\if
	L#1\relax\rlap{\raise.5\dimen7\hbox{\kern.7ex
			\char'47}}#1\relax\else\message{accent \string\soft \space #1 not
		defined!}#1\relax\fi\fi\fi\fi\fi\fi}


\begin{aicauthors}
	\begin{authorinfo}[vd]
		Vida Dujmovi{\'c}\\
		School of Computer Science and Electrical Engineering\\
		University of Ottawa\\
		Ottawa, Canada\\
		\texttt{vida.dujmovic@uottawa.ca}\\
		\url{http://cglab.ca/~vida/} 
	\end{authorinfo}
	\begin{authorinfo}[le]
		Louis Esperet\\ 
		Laboratoire G-SCOP\\
		(CNRS, Univ.\ Grenoble Alpes)\\
		Grenoble, France \\
		\texttt{louis.esperet@grenoble-inp.fr}\\
		\url{https://oc.g-scop.grenoble-inp.fr/esperet/} 
	\end{authorinfo}
	\begin{authorinfo}[gj]
		Gwena\"{e}l Joret\\
		D\'epartement d'Informatique\\ 
		Universit\'e Libre de Bruxelles\\
		Brussels, Belgium \\
		\texttt{gjoret@ulb.ac.be}\\
		\url{http://di.ulb.ac.be/algo/gjoret/}
	\end{authorinfo}
	\begin{authorinfo}[bw]
		Bartosz Walczak\\
		Department of Theoretical Computer Science\\ 
		Faculty of Mathematics and Computer Science\\ Jagiellonian University\\
		Krak\'ow, Poland\\ 
		\texttt{walczak@tcs.uj.edu.pl}\\
		\url{http://www.tcs.uj.edu.pl/walczak}
	\end{authorinfo}
	\begin{authorinfo}[dw]
		David R. Wood\\
		School of Mathematics\\
		Monash   University\\
		Melbourne, Australia\\ 
		\texttt{david.wood@monash.edu}\\
		\url{http://users.monash.edu/~davidwo/}
	\end{authorinfo}
\end{aicauthors}

\end{document}